\documentclass[12pt]{amsart}
\usepackage{ucs}
\usepackage{amssymb}
\usepackage{amsthm}
\usepackage{amsmath}
\usepackage{latexsym}
\usepackage[cp1251]{inputenc}
\usepackage{graphicx}
\usepackage{wrapfig}
\usepackage{caption}
\usepackage{subcaption}
\usepackage{indentfirst}
\usepackage[left=2.5cm,right=2.5cm,top=2.4cm,bottom=2.4cm,bindingoffset=0cm]{geometry}
\usepackage{enumerate}
\usepackage{makecell}

\makeatletter 
\def\@seccntformat#1{\csname the#1\endcsname. } 
\def\@biblabel#1{#1.} 
\DeclareMathOperator{\lcm}{lcm}
\DeclareMathOperator{\gcdd}{gcd}
\DeclareMathOperator{\Aut}{Aut}

\makeatother

\title{ON RECOGNITION OF~${A_6\times A_6}$ BY THE SET OF CONJUGACY CLASS SIZES}

\author{Viktor Panshin}
\address{Novosibirsk State University, Novosibirsk, Russia}
\email{v.pansh1n@yandex.ru}

\date{}
\newtheorem{state}{Statement}[section]

\newtheorem*{cc}{Question}
\newtheorem{lemm}[state]{Lemma}
\newtheorem*{T2}{Theorem}
\begin{document}

	\vspace{\baselineskip}
	\vspace{\baselineskip}
	
	\vspace{\baselineskip}
	
	\vspace{\baselineskip}

	\begin{abstract}
		For a finite group $G$ denote by $N(G)$ the set of conjugacy class sizes of $G$. Recently the following question has been asked: Is it true that for each nonabelian finite simple group $S$ and each $n\in\mathbb{N}$, if the set of class sizes of a finite group $G$ with trivial center is the same as the set of class sizes of the direct power $S^n$, then $G\simeq S^n$? In this paper we approach an answer to this question by proving that $A_6\times A_6$ is uniquely determined by $N(A_6\times A_6)$ among finite groups with trivial center.
	\end{abstract}
	
	\maketitle

	\section{Introduction}
	For a finite group $G$ denote by $N(G)$ the set of conjugacy class sizes of $G$. In $1980$s J.G.~Thompson posed the following conjecture: if $L$ is a finite nonabelian simple group, $G$ is a finite group with $Z(G)=1$ and $N(G)=N(L)$, then $G\simeq L$. Later A.S.~Kondrat'ev added this conjecture to the \emph{Kourovka Notebook} \cite[Question~$12.38$]{Cou}. In a series of papers of different authors, it took more than twenty years to confirm the conjecture. The final step was done in \cite{Gor2}, where a full historical overview of the proof can be found.
	
	There are many ways to generalize Thompson's conjecture, in this paper we consider one of them.
	
	\begin{cc}\cite[Question~$20.29$]{Cou}		
		Let $S$ be a nonabelian finite simple group. Is it true that for any $n\in\mathbb{N}$, if the set of class sizes of a finite group $G$ with $Z(G)=1$ is the same as the set of class sizes
		of the direct power $S^n$, then $G\simeq S^n$?		
	\end{cc}
	
	There are no examples when the answer to this question is negative and the only studied case with $n>1$ is when $L=A_5$ and $n=2$  \cite{Gor3}. We continue studying this question by proving the following theorem.
	
	\begin{T2}
		If $G$ is a group with $Z(G)=1$ and $N(G) = N(A_6\times A_6)$, then $G\simeq A_6\times A_6$.
	\end{T2}
	
	\section{Preliminaries}	
	
	Let $n,m$ be integers and $p$ a prime. By $n_p$ we denote the maximal power of $p$ dividing $n$. The least common multiple and the greatest common divisor of $n$ and $m$ is denoted by lcm$(n,m)$ and gcd$(n,m)$ respectively.
	
	By $\pi(G)$ we denote the set of all prime divisors of the order of $G$. Given an element $x$ of a group $G$, denote by $x^G$ the set $\{x^g$ $|$ $g\in G\}$, that is, the conjugacy class contaning $x$, and by $|x^G|$ its size.
	
	\begin{lemm}\label{N}
		
		${N(A_6)=\{1,40,45,72,90\}}$, and $n\in N(A_6\times A_6)$ if and only if $n=a\cdot b$, where $a,b\in N(A_6)$.
		
		\begin{proof}
			The first statement can be easily verified. The second statement follows from the well-known fact that $N(G\times H)=N(G)\cdot N(H)=\{n~|~n=a\cdot b,~a\in N(G), b\in N(H) \}$.
		\end{proof}
	\end{lemm}

	\begin{lemm}\label{gor}\cite[Theorem~5.2.3]{Gor}
	Let $A$ be a $p^\prime$-group of automorphisms of an abelian $p$\nobreakdash-group~$P$. Then $P=C_P(A)\times[P,A]$.
\end{lemm}

In the following two lemmas we list well-known facts on sizes of conjugacy classes in finite groups (see, e.g.~\cite[Lemma 1]{Gor3}).

\begin{lemm}\label{big1}
	
	Suppose that $G$ is a finite group, $K\unlhd G$, $x\in G$, and $\overline{x}\in \overline{G} = G/K$. Then $|x^K|$ and $|\overline{x}^{\overline{G}}|$ divide $|x^G|$.
\end{lemm}

\begin{lemm}\label{big2}
	
	Let $G$ be a finite group and $x,y\in G$, $xy=yx$, and $\gcdd(|x|,|y|)=1$. Then $C_G(xy)=C_G(x)\cap C_G(y)$, $|x^G|\cdot|y^G|\ge|(xy)^G|$ and $\lcm (|x^G|,|y^G|)$ divides $|(xy)^G|$.
	
\end{lemm}

\begin{lemm}\label{wreath}
	If $G=H\wr \langle g\rangle$ is the wreath product of a group $H$ and involution $g$, then $|g^G|=|H|$.
\end{lemm}
	
	\begin{proof}
		If $B=H\times H$ is the base of the wreath product, then $C_B(g)=\{(h,h)~|~h\in H\}$, so the statement follows.
	\end{proof}

	\begin{lemm}\label{index} 
		Let $x\in G$ with $p^n=|x^G|_p<|G|_p$, where $n\geq1$. Then there is a $p$-element $y\in C_G(x)$ with $|y^G|_p\leq p^{n-1}$. If, in addition, $x$ is a $p'$-element, then $|(xy)^G|_p=|x^G|_p$.
		\begin{proof}
			Take $P\in Syl_p(G)$ such that $C_G(x)\cap P = \tilde{P}\in Syl_p(C_G(x))$, so $|P:\tilde{P}| = p^n$. Put $N=N_P(\tilde{P})$. There is a nontrivial element $y\in Z(N)\cap\tilde{P}$. Since $|P:N|\leq p^{n-1}$, it follows that $|y^G|_p\leq p^{n-1}$.
			
			By the above paragraph, $\tilde{P}\le C_G(x)\cap C_G(y)\le C_G(xy)$. If $x$ is a $p'$-element, then Lemma~\ref{big2} implies that $|x^G|$ divides $|(xy)^G|$, so the second statement of the lemma follows.
		\end{proof}
	\end{lemm}

	\section{Proof of the theorem}	
	
	Let $G$ be a finite group with ${Z(G)=1}$ such that ${N(G)=N(A_6\times A_6)}$. It follows from~\cite[Corollary~1]{Cam} that $\pi(G)=\pi(A_6\times A_6)=\{2,3,5\}$.

	\begin{lemm}\label{5e}
		
		There is a $5$-element $x\in G$ with $|x^G| = 72$ and a $3$-element $u\in G$ with $|u^G|=40$.
		
		\begin{proof}
			Lemma~\ref{N} yields that there is an element $x$ of prime order with $|x^G| = 72$. If $|x|=5$, then we are done, so we can assume that $|x|\in\{2,3\}$.
			
			Suppose that $|x| = 2$. Lemma~\ref{index} implies that there is a 3-element $y\in C_G(x)$ with $|y^G|_3\leq3$, so $|y^G|\in\{40,40^2\}$ due to Lemma~\ref{N}. If $|y^G|=40^2$, then lcm$(|x^G|,|y^G|)=$ lcm$(40^2,72)=120^2$. By Lemma~\ref{big2}, there must be a multiple of $120^2$ in $N(G)$; a contradiction. Hence $|y^G|=40$. 
			
			Applying Lemma~\ref{index} once more, we obtain a $5$-element $z\in C_G(y)$ with $|z^G|_5\leq1$, so $|z^G|\in\{72,72^2\}$. If $|z^G|=72^2$, then $N(G)$ must contain a multiple of $72^2\cdot 5$; a contradiction. Thus, $z$ is a desired $5$-element with  $|z^G|=72$.
			
			The case $|x|=3$ can be treated similarly: for a suitable $2$-element $y\in C_G(x)$, there is a $5$-element $z\in C_G(y)$ with $|z^G|=72$.			 	
			
			Again, Lemma~\ref{N} yields that there is an element $u$ of prime order with $|u^G| = 40$. Applying the same arguments as above we obtain a $3$-element satisfying this property.
		\end{proof}
	\end{lemm}

		
		
			
			
	
	In the next three lemmas we show that $O_p(G)=1$ for every $p\in \pi(G)$; so the socle of $G$ must be the direct product of nonabelian simple groups.

	\begin{lemm}\label{O5}
	
	$O_5(G) = 1$.
	
	\begin{proof}
		
		Suppose the contrary. Let $V$ be a nontrivial normal abelian $5$-subgroup of $G$.
		
		Let $w$ be a $3$-element with $|w^G|=40$. Lemma~\ref{gor} yields $V = C_V(w)\times [V,w]$. Since $w$ acts freely on $[V,w]$, we have that ${|[V,w]|-1}$ is a multiple of $3$, so $|[V,w]|= 5^{2k}$ for $k\geq 0$. In view of Lemma~\ref{big1},			
		$$
		|w^V|=|V:C_V(w)|=|[V,w]|=5^{2k}\text{~divides~} |w^G|_5=5.
		$$
		Thus, $|[V,w]|=1$ and $V=C_V(w)$.
		
		If ${P\in Syl_5(G)}$ and $y\in V\cap Z(P)\setminus \{1\}$, then $|y^G|_5=1$. It follows from Lemma~\ref{N} that $|y^G|\in\{72,72^2\}$. Suppose that $|y^G|=72^2$. Since $\gcdd(|w|,|y|)=1$ Lemma~\ref{big2} implies that lcm$(|w^G|,|y^G|) = 72^2\cdot5$ divides $|(wy)^G|$, which contradicts Lemma~\ref{N}. Hence $|y^G|=72$. 
		
		By Lemma~\ref{index}, there is a $2$-element $z\in C_G(y)$ with $|z^G|_2\leq2^2$, so $|z^G|\in\{45,90,45^2,45\cdot90,90^2\}$. If $|z^G|\in\{45^2,45\cdot90,90^2\}$, then Lemma~\ref{big2} implies that $45^2\cdot2^3$ divides $|(yz)^G|$, a contradiction. Therefore, $|z^G|\in\{45,90\}$. Since $y$ is a $2'$-element, $|(yz)^G|_2=|y^G|_2=2^3$ due to Lemma~\ref{index}. It follows from Lemma~\ref{big2} that lcm$(|y^G|,|z^G|)=2^3\cdot3^2\cdot5=360$ divides $|(yz)^G|$. Thus, $|(yz)^G|\in\{40\cdot45,45\cdot72\}$. 
		
		Applying Lemma~\ref{index} once more, we obtain a $3$-element $x\in C_G(z)$ such that $|x^G|_3\le3$, whence $|x^G|\in\{40,40^2\}$. In fact $|x^G|=40$, otherwise Lemma~\ref{big2} implies that $40^2\cdot3^2$ divides $|(xz)^G|$, which is impossible. We have $|(xz)^G|_3=|z^G|_3=3^2$ by Lemma~\ref{index} and $|(xz)^G|$ is a multiple of lcm$(|x^G|,|z^G|)=360$. Therefore, $|(xz)^G|\in\{40\cdot45,40\cdot72,40\cdot 90\}$.
				
		By the second paragraph $x$ centralizes $V$, in particular, $x\in C_G(y)$. Lemma~\ref{big2} implies that $|(xy)^G|\le|x^G|\cdot|y^G|=40\cdot72$  and $|(xy)^G|$ is a multiple of lcm$(|x^G|,|y^G|)=360$. Therefore, $|(xy)^G|\in\{40\cdot45,40\cdot72\}$.  
		
		Since $x,y,z$ are of coprime orders and centralize each other, $|(xyz)^G|$ is a multiple of lcm$(|(xy)^G|,|(xz)^G|,|(yz)^G|)$. Therefore, we can assume that $|(xy)^G|=|(yz)^G|=40\cdot45$ and $|(xz)^G|\in\{40\cdot45,40\cdot90\}$, because otherwise $|(xyz)^G|\notin N(G)$.			
		
		It is clear that $|(xyz)^G|_5=|(xy)^G|_5$ or, equivalently, $|C_G(xy)\cap C_G(z)|_5=|C_G(xy)|_5$. The latter implies that a Sylow $5$-subgroup $Q$ of $C_G(xy)\cap C_G(z)$ is a Sylow $5$-subgroup of $C_G(xy)$. Since $V$ is a normal $5$-subgroup of $C_G(xy)=C_G(x)\cap C_G(y)$, it follows that $V\le Q\le C_G(z)$.
		
		Suppose that $P$ is a Sylow 5-subgroup of $G$ containing a Sylow $5$-subgroup $\tilde{P}$ of $C_G(x)$ and $u\in V\cap Z(P)\setminus\{1\}$. Then $|u^G|=72$ by the arguments from the third paragraph of the proof. Since $u$ centralizes $\tilde{P}$, it follows that $|(xu)^G|_5=|x^G|_5=5$. Lemma~\ref{big2} and Lemma~\ref{index} yield that $|(xu)^G|=40\cdot 72$. Note that $u\in V\le C_G(z)$, so $|(xuz)^G|$ is a multiple of lcm$(|(xu)^G|,|(xz)^G|)=2^6\cdot3^2\cdot5^2$, a contradiction.
	\end{proof}
\end{lemm}	

The proofs of the next two lemmas are very similar to that of Lemma~\ref{O5}. Nevertheless, we provide them.
	
	\begin{lemm}\label{O3}
	
	$O_3(G) = 1$.
	
	\begin{proof}
		
		Suppose, to the contrary, that $V$ is a nontrivial abelian normal $3$-subgroup of $G$.
		
		By the same arguments as in the proof of Lemma~\ref{O5}, an arbitrary $5$-element $w\in G$ with $|w^G|=72$ centralizes $V$ and there is an element $y\in V$ with $|y^G|=40$.
		
		In view of Lemma~\ref{index}, there is a $2$-element $z\in C_G(y)$ with $|z^G|_2\leq2^2$. Following exactly the proof of the previous lemma, we obtain that $|z^G|\in\{45,90\}$ and $|(yz)^G|\in\{40\cdot45,45\cdot72\}$. 
		
		Applying Lemma~\ref{index} once more, we obtain a $5$-element $x\in C_G(z)$ such that $|x^G|_5\le1$, whence $|x^G|\in\{72,72^2\}$. In fact $|x^G|=72$, otherwise Lemma~\ref{big2} implies that $72^2\cdot5$ divides $|(xz)^G|$, which is impossible. We have $|(xz)^G|_5=|z^G|_5=5$ by Lemma~\ref{index} and $|(xz)^G|$ is a multiple of lcm$(|x^G|,|z^G|)=360$. Therefore, $|(xz)^G|\in\{40\cdot72,45\cdot72,72\cdot 90\}$.
		
		By the second paragraph, $x\in C_G(y)$. It follows that $|(xy)^G|\le|x^G|\cdot|y^G|=40\cdot72$  and $|(xy)^G|$ is a multiple of lcm$(|x^G|,|y^G|)=360$ due to Lemma~\ref{big2}. Therefore, $|(xy)^G|\in\{40\cdot45,40\cdot72\}$.  
		
		The elements $x,y,z$ are of coprime orders and centralize each other. Hence, $|(xyz)^G|$ is a multiple of lcm$(|(xy)^G|,|(xz)^G|,|(yz)^G|)$, which is impossible due to Lemma~\ref{N}.		
	\end{proof}
\end{lemm}

\begin{lemm}\label{O2}
		
	$O_2(G) = 1$.
		
		\begin{proof}
			
			Suppose that $V$ is a nontrivial normal abelian $2$-subgroup of $G$.
						
			Applying the same arguments as in the proof of Lemma~\ref{O5}, we see that an arbitrary $5$-element $w\in G$ with $|w^G|=72$ centralizes $V$ and there is an element $y\in V$ with $|y^G|=45$.
			
			By Lemma~\ref{index}, there is a $3$-element $z\in C_G(y)$ with $|z^G|_3\leq3$, so $|z^G|\in\{40,40^2\}$. If $|z^G|=40^2$, then Lemma~\ref{big2} implies that $40^2\cdot3^2$ divides $|(yz)^G|$, a~contradiction. Therefore, $|z^G|=40$. 
			It follows from Lemma~\ref{big2} that lcm$(|y^G|,|z^G|)=360$ divides $|(yz)^G|$ and $40\cdot 45\ge |(yz)^G|$. Thus, $|(yz)^G|=40\cdot45$.
						
			Applying Lemma~\ref{index} once more, we obtain a $5$-element $x\in C_G(z)$ with $|x^G|_5=1$, so $|x^G|\in\{72,72^2\}$. If $|x^G|=72^2$, then Lemma~\ref{big2} implies that $72^2\cdot5$ divides $|(xz)^G|$, a contradiction. Therefore, $|x^G|=72$.  Moreover, Lemma~\ref{index} yields $|(xz)^G|_5=|z^G|_5=5$. We have that lcm$(|x^G|,|z^G|)$ divides $|(xz)^G|$ and  $40\cdot 72\ge |(xz)^G|$ due to Lemma~\ref{big2}. Therefore, $|(xz)^G|=40\cdot72$.
			
			By the second paragraph, $x\in C_G(y)$. It follows that $x,y,z$ are of coprime orders and centralize each other. Hence, $|(xyz)^G|$ is a multiple of lcm$(|(xz)^G|,|(yz)^G|) = 40\cdot72\cdot5$, a~contradiction.
		\end{proof}
	\end{lemm}

	
	Thus, $M\leq G\leq\Aut(M)$, where $M=S_1\times\dots \times S_k$ is a direct product of finite nonabelian simple groups $S_1,\dots, S_k$. Since $\pi(S_i)\subseteq\pi(G)$, it follows that $S_i\in\{A_5,A_6,U_4(2)\}$. We will use information about this groups from~\cite{atlas}.
	
	\begin{lemm}\label{almost}
		
		 $M \simeq A_6\times A_6$.
		
		\begin{proof} 
			If $k = 1$, then $|G|_5>|\Aut(M)|_5$, a contradiction. If $k\geq 3$, then there is an element $x\in M$ with $|x^M|_5=5^3$. Lemma~\ref{big1} implies that $|x^M|$ divides $|x^G|$, a contradiction.
			
			Therefore, $k=2$. If one of the two direct factors is isomorphic to $U_4(2)$, then there is an element $x\in M$ with $|x^M|_2>2^6$, which is impossible.
			
			Thus, we have ${M\in\{A_5\times A_5, A_5\times A_6, A_6\times A_6\}}$. If ${M\in\{A_5\times A_5,  A_5\times A_6\}}$, then $|G|_3>|\Aut(M)|_3$, a contradiction. So $M \simeq A_6\times A_6$, as claimed.	
		\end{proof}
	\end{lemm}
	
	\begin{lemm}
		
		$G\simeq A_6\times A_6$.
		
		\begin{proof}
			By Lemma~\ref{almost}, we have $S_1\times S_2\leq G\leq\Aut(S_1\times S_2) = \Aut(S_1)\wr \langle a\rangle$, where $a$ is an involution and $S_1\simeq S_2\simeq A_6$.
			
			If $a\in G$, then $G\simeq H\wr \langle a\rangle$, where $A_6\leq H\leq\Aut(A_6)$. Lemma~\ref{wreath} states that $|a^G|=|H|$. This is a contradiction since $|H|\not\in N(G)$.
			
			Suppose that $a\not\in G$, so $S_1\times S_2\leq G\leq\Aut(S_1)\times\Aut(S_2)$. If $S_1\times S_2< G$, then $G$ includes a subgroup $G_1$ of the form $(S_1\times S_2).2=\langle S_1\times S_2, (\varphi_1,\varphi_2)\rangle$, where $\varphi_i\in\Aut(S_i)$. We may assume that $\varphi_1\not\in S_1$. Also observe that $G_1$ is normal in $G$.
			
			There are three isomorphism types of $A_6.2$. For every possible isomorphism type of $A_6.2$ there is an element $x\in A_6$ with $|x^{A_6.2}|_2>2^3$, so $C_{A_6.2}(x)=C_{A_6}(x)$.

			Take an element $x=(x_1,x_2)\in G_1$ such that $|{x_1}^{\langle S_1.\varphi_1\rangle}|>2^3$ and $|{x_2}^{\langle S_2.\varphi_2\rangle}|\ge2^3$. Since $C_{\langle S_i,\varphi_i\rangle}(x_i)=S_i$ regardless of whether $\varphi_i$ is outer or not, we have $C_{G_1}(x) = C_{S_1}(x_1)\times C_{S_2}(x_2)$. It follows that $|x^G_1|_2>2^6$. This contradiction completes the proof of the lemma and the~theorem.
		\end{proof}
	\end{lemm}
	
	The author is grateful to I.B.~Gorshkov, M.A. Grechkoseeva and A.V.~Vasil'ev for helpful comments and suggestions.

\end{document}